\documentclass[12pt]{article}
\usepackage{amsmath, amsthm, amssymb, mathrsfs}
\usepackage{graphicx}
\usepackage[all,2cell]{xy}
\UseAllTwocells
\usepackage[all]{xy}
\usepackage{hyperref}
\usepackage{graphicx}
\usepackage{manfnt}

\def\bA{{\bf A}}
\def\bB{{\bf B}}

\def\bG{{\bf G}}
\def\bH{{\bf H}}

\def\bM{{\bf M}}
\def\bN{{\bf N}}

\def\bS{{\bf S}}
\def\bT{{\bf T}}

\def\bZ{{\bf Z}}

\def\aq{/  \kern-.25em / }

\def\g{\mathfrak{g}}

\def\cH{\mathcal{ H}}

\def\Z{\mathbb{ Z}}

\def\boxit#1{\vbox{\hrule\hbox{\vrule\kern3pt
          \vbox{\kern3pt#1\kern3pt}\kern3pt\vrule}\hrule}}

\begin{document}

\newtheorem{theorem}{Theorem}[subsection]
\newtheorem{lemma}[theorem]{Lemma}
\newtheorem{proposition}[theorem]{Proposition}
\newtheorem{problem}[theorem]{Problem}
\newtheorem{corollary}[theorem]{Corollary}

\theoremstyle{definition}
\newtheorem{definition}[theorem]{Definition}
\newtheorem{example}[theorem]{Example}
\newtheorem{xca}[theorem]{Exercise}

\theoremstyle{remark}
\newtheorem{remark}[theorem]{Remark}

\def\goth{\frak}

\def\GL{{\rm GL}}
\def\tr{{\rm tr}\, }
\def\A{{\Bbb A}}
\def\bs{\backslash}
\def\Q{{\Bbb Q}}
\def\R{{\Bbb R}}
\def\Z{{\Bbb Z}}
\def\C{{\Bbb C}}
\def\SL{{\rm SL}}
\def\cS{{\cal S}}
\def\cH{{\cal H}}
\def\G{{\Bbb G}}
\def\F{{\Bbb F}}
\def\cF{{\cal F}}

\def\cB{{\cal B}}
\def\cA{{\cal A}}
\def\cE{{\cal E}}

\newcommand{\oB}{{\overline{B}}}
\newcommand{\oN}{{\overline{N}}}

\def\CC{{\Bbb C}}
\def\ZZ{{\Bbb Z}}
\def\QQ{{\Bbb Q}}
\def\cS{{\cal S}}

\def\Ad{{\rm Ad}}

\def\bG{{\bf G}}
\def\bH{{\bf H}}
\def\bT{{\bf T}}
\def\bM{{\bf M}}
\def\bB{{\bf B}}
\def\bN{{\bf N}}
\def\bS{{\bf S}}
\def\bZ{{\bf Z}}
\def\t{\kern.1em {}^t\kern-.1em}
\def\cc#1{C_c^\infty(#1)}
\def\Fx{F^\times}
\def\half{\hbox{${1\over 2}$}}
\def\T{{\Bbb T}}
\def\Ox{{\frak O}^\times}
\def\Ex{E^\times}
\def\Ecl{{\cal E}}

\def\g{{\frak g}}
\def\h{{\frak h}}
\def\k{{\frak k}}
\def\ft{{\frak t}}
\def\n{{\frak n}}
\def\b{{\frak b}}

\def\2by2#1#2#3#4{\hbox{$\bigl( 
{#1\atop #3}{#2\atop #4}\bigr)$}}
\def\wh{\Xi}
\def\C{{\Bbb C}}
\def\bs{\backslash}
\def\adots{\mathinner{\mkern2mu
\raise1pt\hbox{.}\mkern2mu
\raise4pt\hbox{.}\mkern2mu
\raise7pt\hbox{.}\mkern1mu}}

\def\tim{\bf}
\def\timit{\it}
\def\timsm{\scriptstyle}
\def\secttt#1#2{\vskip.2in\noindent
{$\underline{\hbox{#1}}$\break (#2)}\vskip.2in}
\def\sectt#1{\vskip.2in\noindent
{$\underline{\hbox{#1}}$}\vskip.2in}
\def\sectm#1{\vskip.2in\noindent
{$\underline{#1}$}\vskip.2in}
\def\mat#1{\left[\matrix{#1}\right]}
\def\cc#1{C_c^\infty(#1)}
\def\ds{\displaystyle}
\def\Hom{\mathop{Hom}\nolimits}
\def\Ind{\mathop{Ind}\nolimits}
\def\bs{\backslash}
\def\ni{\noindent}
\def\eb{{\bf e}}
\def\fb{{\bf f}}
\def\hb{{\bf h}}
\def\rg{{\goth R}}
\def\ig{{\goth I}}
\def\ccl{{\cal C}}
\def\dcl{{\cal D}}
\def\ecl{{\cal E}}
\def\hcl{{\cal H}}
\def\ocl{{\cal O}}
\def\ncl{{\cal N}}
\def\ag{{\goth a}}
\def\bg{{\goth b}}
\def\cg{{\goth c}}
\def\og{{\goth O}}
\def\hg{{\goth h}}
\def\lg{{\goth l}}
\def\mg{{\goth m}}
\def\Og{{\goth O}}
\def\rg{{\goth r}}
\def\sg{{\goth s}}
\def\Sg{{\goth S}}
\def\tg{{\goth t}}
\def\zg{{\goth z}}
\def\C{{\Bbb C}}
\def\Q{{\Bbb Q}}
\def\R{{\Bbb R}}
\def\T{{\Bbb T}}
\def\Z{{\Bbb Z}}
\def\t{{}^t\kern-.1em}
\def\tr{\hbox{tr}}
\def\ad{{\rm ad}}
\def\Ad{{\rm Ad}}
\def\2by2#1#2#3#4{\hbox{$\bigl( {#1\atop #3}{#2\atop #4}\bigr)$}}

\def\Kcl{{\cal K}}
\def\Pcl{{\cal P}}
\def\Scl{{\cal S}}
\def\Vcl{{\cal V}}
\def\Wcl{{\cal W}}
\def\Ecl{{\cal E}}
\def\A{{\Bbb A}}
\def\F{{\Bbb F}}
\def\T{{\Bbb T}}
\def\go{{\goth O}}
\def\Ox{{\goth O}^\times}
\def\Fx{F^\times}
\def\Ex{E^\times}
\def\half{\hbox{${1\over 2}$}}
\def\vtwo{\vskip .2in}
\def\BibFH{{\bf [1]}}
\def\BibGod{{\bf [2]}}
\def\BibCrelle{{\bf [3]}}
\def\BibAmJ{{\bf [4]}}
\def\BibDuke{{\bf [5]}}
\def\BibJL{{\bf [6]}}
\def\BibRR{{\bf [7]}}

\def\BibFH{{\bf [1]}}
\def\BibGod{{\bf [2]}}
\def\BibCrelle{{\bf [3]}}
\def\BibAmJ{{\bf [4]}}
\def\BibDuke{{\bf [5]}}
\def\BibJL{{\bf [6]}}
\def\BibRR{{\bf [7]}}

\newcommand{\N}{\mathbb{N}}

\newcommand{\gen}{{\operatorname{gen}}}
\newcommand{\ind}{\operatorname{ind}}
\newcommand{\Wh}{\mathcal{W}}
\newcommand{\Kr}{\mathcal{K}}
\newcommand{\V}{\mathcal{V}}
\newcommand{\U}{\mathcal{U}}
\renewcommand{\O}{\mathcal{O}}
\newcommand{\lift}{\goth{g}}
\newcommand{\inv}{\iota}
\newcommand{\supp}{{\operatorname{Supp}}}
\def\cW{{\cal W}}

\title{The Construction of Regular\\  Supercuspidal Representations}
\author{Jeffrey Hakim}
\date{\today}
\maketitle
\abstract{This paper simplifies and further develops various aspects of Tasho Kaletha's construction of regular supercuspidal representations.  Moreover, Kaletha's construction is connected with the author's revision of Yu's construction of tame supercuspidal representations.  This allows for a more direct construction of regular supercuspidal representations that is more amenable to applications.}
\tableofcontents

\parskip=.13in

\section{Introduction}

The theory of regular supercuspidal representations was introduced by Tasho Ka\-le\-tha in \cite{KalYu}.  Given a suitable  character $\mu$ of an elliptic maximal torus in a $p$-adic connected, reductive group, Kaletha constructs a regular supercuspidal representation $\pi (\mu)$ by first associating to $\mu$ a technical object $\Psi_\mu$ known as a generic, cuspidal $G$-datum (see \cite[Definition 3.11]{MR2431732}) and then applying Jiu-Kang Yu's construction \cite{MR1824988} of tame supercuspidal representations.

The purpose of this paper is to construct Kaletha's correspondence in a simpler and more direct way using the variant of Yu's construction in \cite{ANewYu} as well as various other technical refinements.
In a sequel to this paper \cite{RegDist}, we  provide evidence that this new approach can significantly facilitate the development of applications.  The results of \cite{RegDist} involve the theory of distinguished regular supercuspidal representations.  Compared to the theory in \cite{MR2431732}, the results in \cite{RegDist} are more precise and the proofs are considerably simpler.

In order to manufacture the $G$-datum $\Psi_\mu$, Kaletha develops  a theory of Howe factorizations  \cite[\S3.7]{KalYu} that generalizes Roger Howe's $\GL_n$  factorization theory.   
In the present paper, we use a weaker factorization (see Definition \ref{weakfact})  that essentially consists of the depth zero factor in the Kaletha's  factorization together with the product of the remaining factors.   (See  \cite[Corollary to Lemma 11]{MR0492087} and  \cite[\S4.3]{MR2431732} for Howe's theory and its adaptation to Yu's construction.)

One of the main points of our revision of Yu's construction in \cite{ANewYu} is that Howe factorizations are non-canonical objects, a fact that can significantly complicate the development of applications of Yu's construction.  In particular, the paper \cite{ANewHM} applies the  theory in \cite{ANewYu} to  simplify the theory of distinguished tame supercuspidal representations in \cite{MR2431732}.  
Though the weak factorizations in this paper are  not canonical, they are nearly canonical and they preserve most of the simplifications from \cite{ANewYu} while offering a  convenient way to reduce to the depth zero case.

Let us now sketch the structure of this paper.  
We are interested in the regular supercuspidal representations of a group $G = \bG(F)$, where $\bG$ is a connected reductive group that is defined over a field $F$ that is a finite extension of a field $\Q_p$ of $p$-adic numbers  (with the same restrictions on $p$ as in \cite{KalYu}).  
These representations are attached to certain pairs $(\bT,\mu)$ that Kaletha calls ``tame, elliptic, regular pairs.''  Here,  $\bT$ is a suitable elliptic, maximal $F$-torus of $\bG$, and $\mu$ is a suitable character of $T= \bT(F)$.  (See \cite[Definition 3.6.5]{KalYu} and Definition \ref{terp} below.)  Implicit in the existence of such pairs is the assumption that $\bG$ must split over a tamely ramified extension of $F$.

In \S\ref{sec:extinv}, we show how to extract various invariants from $\mu$ that are needed in the construction of the associated regular supercuspidal representation $\pi (\mu)$ of $G$.  These invariants nearly form a generic cuspidal $G$-datum, except that there is no Howe factorization.  For positive depth regular supercuspidal representations, in place of the Howe factorizations (as in \cite[\S3.7]{KalYu}) we use the  ``weak factorizations''  defined in \S\ref{sec:zeroreduction}.  A weak factorization is a pair $(\mu_- , \mu_+)$ of characters, where $\mu_-$ is a depth zero character of $T$ and $\mu_+$ is a positive depth character of $H = \bH (F)$, with $\bH$ being an $F$-subgroup of $\bG$ that becomes a Levi subgroup over  some extension of $F$.  Among other things, it is required that  $\mu (t)= \mu_-(t) \mu_+(t)$  for all $t\in T$.

For those familiar with Yu's construction, we mention that $\bH$ is essentially Yu's group $\bG^0$.
In \S\ref{sec:simplified}, the weak factorization is used to define a certain representation 
$$\rho_T^\mu = \rho_T^{\mu_{-}}\otimes   (\mu_+\,  |\,  TH_{x,0})$$
of a  compact-mod-center subgroup $TH_{x,0}$.  The group $H_{x,0}$ is the analogue of Yu's group $\bG^0( F)_{x,0}$. It is a maximal parahoric subgroup in $H$.  Note that $\rho^\mu_T$ is independent of the choice of weak factorization.

For those familiar with Kaletha's construction, we note that $\rho_T^{\mu_{-}}$ is essentially Kaletha's representation $\tilde\kappa_{(S,\theta)}$, however, we give a more direct definition (in \S\ref{sec:simplified}).  According to Lemma 3.4.12 \cite{KalYu}, the representation $\rho_T^{\mu_-}$ induces an irreducible representation of the stabilizer $H_x$ in $H$ of the vertex $x$ in the reduced building of $H$.  A similar proof shows that $\rho_T^\mu$ induces an irreducible representation $\rho$ of $H_x$.
(See \S3.3 \cite{ANewYu}.)

We show in Lemma \ref{permiss} that $\rho$ is a permissible representation in the sense of \cite[Definition 2.1.1]{ANewYu}.  
Given a permissible representation, such as our $\rho$, the construction in \cite{ANewYu} produces a corresponding supercuspidal representation $\pi (\rho)$.  When $\rho$ is induced from $\rho_T^\mu$, the representation $\pi (\rho)$ is equivalent to Kaletha's representation $\pi (\mu)$.
In sections \ref{sec:first} and \ref{sec:second}, we compare the details of our construction with Kaletha's.

In \S\ref{sec:rhoTmu}, we prove a formula (Proposition \ref{DLformula}) for the character of $\rho_T^\mu$. 
 In formulating the statement of our character  formula, we have made an extra effort to give a statement that  is as simple as possible and most closely resembles the Deligne-Lusztig (virtual) character formula.  This is partly due to the fact that this is what works best for the applications in \cite{RegDist} (which require that we average the character over the fixed points of an involution). 

 Our character formula is similar to Proposition 3.4.14 \cite{KalYu}, which treats the depth zero case, but we have revised Kaletha's treatment in several ways.  For example, in \S\ref{sec:topJoe}, we customize the theory of topological Jordan decompositions (as in \cite{MR2408311}) for our applications to compact-mod-center subgroups.  
 
 The character of $\rho_T^\mu$ may be viewed as an extension of a Deligne-Lusztig character, and, consequently, the proof of the character formula entails a modest extension of the Deligne-Lusztig character formula.  Rather than simply describing the changes required to Deligne-Lusztig's proof, we have included a complete proof.

Our treatment of regular supercuspidal representations provides a new perspective (discussed in \S\ref{sec:first}) on what Yu refers to as the ``inductive structure'' of his construction.  Recall that Yu actually  associates a sequence $\pi_0,\dots ,\pi_d$ of supercuspidal representations to a generic cuspidal $G$-datum $\Psi$, not simply the single representation $\pi (\Psi) = \pi_d$ of the given group $G$.   
(See \cite[Introduction]{MR1824988}
and
\cite[\S4.2]{MR2431732}.)  This sequence depends on the choice of factorization or, in other words,  the representations in the sequence depend partly on the component $\vec\phi$ in Yu's $G$-datum.   In our approach, one has a different, but related, sequence in which  each of the representations in the sequence is a canonical invariant of $\pi (\mu)$.

Finally, we stress that this paper is motivated by applications to distinguished representations in \cite{RegDist}.
If $\pi$ is a regular supercuspidal representation of $G$ and $\theta$ is an involution of $G$, we compute in \cite{RegDist} the dimension of ${\rm Hom}_{G^\theta}(\pi, 1)$, where $G^\theta$ is the group of fixed points of $\theta$.  To do this, we first need a reduction to the study of when $\rho_T^\mu$ is distinguished.  Then we need to average the character of $\rho_T^\mu$ over the fixed points of $\theta$.  In both of these steps, the theory in this paper greatly reduces the effort needed to compute the desired dimensions.  We expect that our theory will also lead to simplifications in other applications that involve regular supercuspidal representations.

\section{The construction of regular supercuspidal representations}

\subsection{Comments on notations}

Assume, for the rest of the paper, that we have fixed a finite extension $F$ of $\Q_p$, and a connected reductive $F$-group $\bG$.  As in \cite{KalYu}, we assume $p$ is not 2 or a bad prime for $\bG$.  Also, as in \cite{KalYu}, we assume $p$ does not divide the order of the fundamental group $\pi_1(\bG_{\rm der})$ of the derived group $\bG_{\rm der}$ of $\bG$, but  this condition may be removed using the theory of $z$-extensions.  This is discussed in  \cite[\S3.9]{KalYu} and \cite[\S3.1]{ANewYu}, and we will also discuss it below in \S\ref{sec:zeroreduction}.

This paper heavily depends on both \cite{ANewYu} and \cite{KalYu}, and, practically speaking, the reader should expect to have both of the latter papers available while reading this paper.  For the most part, we follow the notations of \cite{ANewYu}, which tends to be consistent with \cite{MR2431732} and \cite{MR1824988}.  This applies, in particular, to the many notations associated with Moy-Prasad filtrations.  
In some cases, we follow the conventions of  \cite{KalYu}.  For example, our Moy-Prasad groups are attached to points in  reduced Bruhat-Tits buildings, rather than extended buildings.
We also use the terminology ``character,'' as opposed to ``quasi-character,'' for a smooth 1-dimensional complex representation.

In general, we use boldface letters for $F$-groups and non-boldface for the groups of $F$-rational points.
For example, $G = \bG (F)$.

Assume we have fixed an algebraic closure $\overline{F}$ of $F$ throughout the paper, and that all extensions of $F$ considered are subfields of $\overline{F}$.  Let $F^{\rm un}$ denote the maximal unramified extension of $F$ contained in $\overline{F}$.
Given a finite extension $F'$ of $F$,  let $\mathfrak{f}_{F'}$ denote the residue field of $F'$.

\subsection{Extracting invariants}\label{sec:extinv}

Assume  $\bT$ is a   maximal $F$-torus in $\bG$ and  $\mu : T\to \C^\times$ is an arbitrary (smooth) character
of  $T$.  We  attach to $\mu$ various invariants, some of which require additional restrictions on $\bT$.

The first invariant is an $F$-subgroup $\bH = \bH_\mu$ of $\bG$ that is a Levi subgroup over the algebraic closure $\overline{F}$.  It is constructed as follows.

Fix $\overline F$ and let $\Gamma = {\rm Gal}(\overline F/F)$.
Let $\Phi$ be the root system $\Phi (\bG ,\bT)$.
Then $\Gamma$ acts on roots via $\gamma \cdot a = \gamma \circ a \circ \gamma^{-1}$.
Given a $\Gamma$-orbit $\mathscr{O}$ in $\Phi$, let $\bT_{\mathscr{O}}$ be the subgroup generated by the tori $\bT_a = {\rm image}(\check a)$ as $a$ varies over $\mathscr{O}$.
Then $\bT_{\mathscr{O}}$ is an $F$-torus (see \cite[Corollary 2.2.7, Corollary 3.2.7]{MR1642713}) and we let $T_{\mathscr{O}} = \bT_{\mathscr{O}}(F)$.

\begin{definition}
Let $\mu$ be a character of $T$.  Define $$\Phi_\mu=\bigcup_{{{\mathscr{O}\in \Gamma\bs \Phi}\atop{ \mu | (T_{\mathscr{O}})_{0+}=1}}} \mathscr{O}.$$ 

Here, $(T_{\mathscr{O}})_{0+}$ denotes the subgroup of positive depth elements in $T_{\mathscr{O}}$.
According to \cite[Lemma 3.7.8]{KalYu} and \cite[Lemma 3.5.1]{ANewYu}, $\Phi_\mu$ is a root subsystem of $\Phi$.  
 Define $\bH = \bH_\mu$  to be the unique $\overline{F}$-Levi subgroup of $\bG$ that contains $\bT$ and has root system $\Phi_\mu$.
\end{definition}

More generally, if one replaces $0+$ by  $r+$ in the definition of $\bH_\mu$, for any nonnegative number $r$, then, as $r$ varies,  one obtains a tower 
$$\bH = \bG^0 \subsetneq \cdots \subsetneq \bG^d = \bG$$ of Levi subgroups of $\bG$.
Our first invariant associated to a character $\mu$ of $T$ is the sequence
$$\vec\bG =(  \bG^0,\dots , \bG^d).$$

For the next invariant, we need to assume $\bT$ is tame and elliptic.  (In other words, $\bT$ splits over a tamely ramified extension of $F$, and $\bT $ is anisotropic modulo the center $Z(\bG)$ of $\bG$.)
Let $E$ be the splitting field for $\bT$ over $F$.  Then $E/F$ is a finite, Galois, tamely ramified extension.  (See \cite[Proposition 3.2.12(i), Proposition 13.2.2(i)]{MR1642713}.)  Over $E$, the torus $\bT$ defines an apartment in the reduced building of $\bG$.  We let $\bG_{\rm der}$ denote the derived group of $\bG$, and we also take $\bT_{\rm der} = \bT\cap \bG_{\rm der}$.  Then $\bT\cap \bG_{\rm der}$ is connected and, in fact, a maximal torus in $\bG_{\rm der}$ that is defined and anisotropic over $F$.  
(See  \cite[Proposition 8.1.8(iii)]{MR1642713}.) The apartment of $\bT$ in the reduced building of $\bG (E)$ is the same as the apartment of $\bT\cap \bG_{{\rm der}}$ in the extended building of $\bG_{{\rm der}}(E)$  or, in symbols, $$\mathscr{A}_{\rm red}(\bG ,\bT,E) = \mathscr{A}(\bG_{\rm der},\bT\cap \bG_{\rm der},E).$$

Next, we observe that since $\bT\cap \bG_{\rm der}$ is $F$-anisotropic, there are no  nontrivial cocharacters in $X_*(\bT\cap \bG_{\rm der})$ that are fixed by ${\rm Gal}(E/F)$, and thus
$$\mathscr{A}_{\rm red}(\bG ,\bT,E)^{{\rm Gal}(E/F)}$$ consists of a single point.

Let 
$x = x_{\bT}$ denote this point.

According to a result in Guy Rousseau's thesis (that is the focus of \cite{MR1871292}), we have an identification
$$\mathscr{B}_{\rm red}(\bG ,F) = \mathscr{B}_{\rm red}(\bG,E)^{{\rm Gal}(E/F)},$$ since $E/F$ is tamely ramified.
So, in fact, the point $x$ lies in $\mathscr{B}_{\rm red}(\bG,F)$ and we have
$$\mathscr{B}_{\rm red}(\bG,F)\cap \mathscr{A}_{\rm red}(\bG,\bT,E) = \{ x\}.$$

Note that $T$ is contained in the stabilizer $G_x$ of $x$ in $G$.

Above, we considered the characters $\mu | (T_{\mathscr{O}})_{0+}$.  Consider now the {\it positive} numbers that occur as depths of these characters.  Listing these numbers in order, we obtain a sequence
$$r_0<r_1<\dots <r_d.$$  (We do not mean to suggest that all of the latter characters have positive depth.)
Let $$\vec r = (r_0,\dots , r_d).$$

Once one has the invariants $\vec\bG$, $x$, and $\vec r$ associated to $(\bT,\mu)$,  a plethora of important  subgroups of $G$ can be defined using Moy-Prasad filtrations and the Bruhat-Tits theory of concave functions, exactly as in \cite{MR1824988}.  
Of primary importance are the subgroups $K_+\subset K$ defined in sections 2.5 and 2.6 of \cite{ANewYu}.
Since we are (mostly) following the notational conventions of \cite{ANewYu}, we refer to \S2  and Lemma 3.9.1 in \cite{ANewYu} for the definitions of the other relevant subgroups.

As discussed above, $K$ is a compact-mod-center subgroup of $G$ such that $\pi (\mu)$ is induced from a certain irreducible representation $\kappa$ of $K$.  (It is essentially the same inducing subgroup used by Yu.)

The subgroup $K_+$ is a compact subgroup of $K$ such that 
$$\kappa |K_+ = \hat\phi \cdot {\rm Id},$$ for a certain  character $\hat\phi$ of $K_+$ defined in \cite[\S3.9]{ANewYu}.

For readers needing some extra intuition regarding $\hat\phi$, we offer the following  heuristic.  The inducing representation $\kappa$ amalgamates various Heisenberg representations, as well as other things.  In restricting to $K_+$, one is restricting to the centers of the various relevant Heisenberg groups.  So, roughly speaking, $\hat\phi$ amalgamates the central characters of the various Heisenberg representations.

The following definition is equivalent to Definition 3.6.5 in \cite{KalYu}  (according to Fact 3.4.1 \cite{KalYu}):

\begin{definition}\label{terp}
A pair $(\bT,\mu)$ is called {\bf a tame, elliptic, regular pair} if:
\begin{itemize}
\item[(1)] $\bT$ is a tame,  elliptic, maximal $F$-torus in $\bG$.
\item[(2)] $\mu$ is a character of $T$ such that $\bT$ is a maximally unramified subtorus of $\bH = \bH_\mu$.
\item[(3)] Any element of $H$ that normalizes $\bT$ and fixes $\mu |T_{0}$ must lie in $T$.
\end{itemize}
\end{definition}

We have already given some indication of the consequences of Condition (1).

Regarding Condition (2), we refer to \S3.4.1 \cite{KalYu} for a discussion of maximally unramified tori.  (See also \S2 \cite{MR1824988}.)  Lemma 3.4.2 \cite{KalYu} implies that if Conditions (1) and (2) hold then $x = x_{\bf T}$ must in fact be a vertex in $\mathscr{B}_{\rm red}(\bH,F)$.  Here, we view $\mathscr{B}_{\rm red}(\bH,F)$ as a subset of $\mathscr{B}_{\rm red}(\bG ,F)$.  What is perhaps more relevant is that there is a natural identification of the apartments associated to $\bT$ in the reduced buildings of $\bH$ and $\bG$ over $E$.
 
Condition (3) is called ``the regularity condition.''  When $\mu$ has depth zero, Condition (3) is the same as the condition that the character of $\mathsf{T}(\mathfrak{f}_F) = T_{0:0+}$ associated to $\mu$ is in general position and thus parametrizes an irreducible, cuspidal Deligne-Lusztig representation of $\mathsf{H}_x^\circ (\mathfrak{f}_F) = H_{x,0:0+}= G_{x,0:0+}$.  (See \cite[Fact 3.4.11]{KalYu}.)  When $\mu$ has positive depth, similar remarks apply to the depth zero component $\mu_{-1}$ of any Howe factorization of $\mu$.  
The fact that $x$ must be a vertex in $\mathscr{B}_{\rm red}(\bH ,F)$ is essential in the depth zero theory.  (See Proposition 6.8 \cite{MR1371680}, as well as \cite[Remark 3.7]{MR1824988} and \cite[\S3.4]{KalYu}.)

\subsection{Reduction to depth zero}\label{sec:zeroreduction}

It is  convenient to suitably factor a given character $\mu$ of $T$ into a depth zero piece and a positive depth piece.  This allows us to structure our arguments as reductions to the depth zero case.

Recall that we are assuming that the residual characteristic $p$ of $F$ divides the order of the fundamental group $\pi_1(\bG_{\rm der})$ of the derived group $\bG_{\rm der}$ of $\bG$.  
If we did not make this assumption,  there would be a technical obstruction to such a reduction-to-depth-zero strategy that could be addressed with the theory of $z$-extensions  as follows.  Let $$1\to \bN\to \bG^\sharp \to \bG\to 1$$ be a $z$-extension of $\bG$, as in \cite[\S3.1]{ANewYu}.  Then $G \cong G^\sharp /N$ and hence a representation of $G$ may be regarded as a representation of $G^\sharp$ with trivial restriction to $N$.  
But $\pi_1(\bG^\sharp_{\rm der})$ is trivial since $\bG^\sharp_{\rm der}$ is simply-connected. So we may as well assume the order of $\pi_1(\bG_{\rm der})$ is not divisible by $p$, since this becomes true after passing to a $z$-extension.

Note also that  if $p$ does not divide the order of $\pi_1(\bG_{\rm der})$ then it also does not divide the order of $\pi_1(\bH_{\rm der})$.  (See \S3.1 and Remark 2.1.2 in \cite{ANewYu} for more details on these matters.)

We make use of the following weak substitute for  Kaletha's Howe factorization:

\begin{definition}\label{weakfact}
If $(\bT,\mu)$ is a tame, elliptic, regular pair for $\bG$ then a {\bf weak  factorization} of $\mu$ is a pair $(\mu_-,\mu_+)$ consisting of a character $\mu_-$ of $T$ and a character $\mu_+$ of $H_\mu$ such that
\begin{itemize}
\item[$\bullet$] $(\bT, \mu_-)$ is a depth zero tame, elliptic, regular pair for  $\bH_\mu$.
\item[$\bullet$]  $\mu = \mu_- (\mu_+|T)$.
\end{itemize}
\end{definition}

The existence of weak  factorizations follows from the existence of Kaletha's Howe factorizations, as we show in the proof of the following:

\begin{lemma} 
Every tame, elliptic, regular pair admits a weak  factorization.
\end{lemma}

\begin{proof}
Suppose $(\bT,\mu)$ is a tame, elliptic, regular pair for $\bG$.  Then, according to Proposition 3.7.4 \cite{KalYu}, $\mu$ admits a Howe factorization $\mu_{-1},\dots, \mu_d$, in the sense of Definition 3.7.1 of \cite{KalYu} with the additional property that $(\bT, \mu_{-1})$ is a depth zero tame, elliptic, regular pair.  Taking $\mu_-= \mu_{-1}$ and $$\mu_+ = \prod_{i=0}^d \mu_i|H$$ gives the required weak  factorization of $\mu$.
\end{proof}

The notations ``$\mu_-$'' and ``$\mu_+$'' reflect that $\mu_+$ contains the positive depth content of $\mu$, while $\mu_-$ can be viewed as the $\mu_{-1}$ factor in a Howe factorization.  As with Howe factorizations, our weak factorizations are not unique.  (This is discussed in more detail in \S3.3 \cite{ANewYu}.)

\begin{definition}\label{musharp}
Given a tame, elliptic, regular pair $(\bT,\mu)$ for $\bG$ and a weak  factorization $(\mu_-,\mu_+)$,  we define characters $\mu_-^\sharp$, $\mu_+^\flat$, and $\mu^\sharp$ of $TH_{x,0+}$ by:
\begin{itemize}
\item[$\bullet$] $\mu_-^\sharp$ is the unique character of $TH_{x,0+}$ that coincides with $\mu_-$ on $T$ and is trivial on $H_{x,0+}$.
\item[$\bullet$] $\mu_+^\flat$ is the restriction of $\mu_+$ to $TH_{x,0+}$.
\item[$\bullet$] $\mu^\sharp = \mu_-^\sharp \mu_+^\flat$.
\end{itemize}
\end{definition}

\begin{lemma}\label{mushlemma}
In Definition \ref{musharp}, the character $\mu^\sharp$ is an extension of $\mu$ from $T$ to $TH_{x,0+}$ that is independent of the choice of weak  factorization.\end{lemma}

\begin{proof}
It is obvious that $\mu^\sharp$ agrees with $\mu$ on $T$.  So it suffices to show that the restriction of $\mu^\sharp$ to $H_{x,0+}$ is independent of the choice of weak Howe factorization.    Suppose we have two weak factorizations, $(\mu_-,\mu_+)$ and $(\dot\mu_- , \dot\mu_+)$.
Then the characters $\mu_+$ and $\dot\mu_+$ both coincide with $\mu$ on $T_{0+}$.
So the character $\mu_+^{-1}\dot\mu_{+}$ of $H$ must be trivial on $T_{0+}$.
Consequently, it must be trivial on $H_{x,0+}$, according to Lemmas 3.2.1 and 3.4.5 of \cite{ANewYu}.  (Here, we are using the fact that $p$ does not divide the order of $\pi_1(\bH_{\rm der})$ and the fact that $\mu_+^{-1}\dot\mu_{+}$ is trivial on $[H,H]\cap H_{x,0+}$.)
Our claim now follows.\end{proof}

\begin{corollary}
In Definition \ref{musharp},
the restrictions of the characters $\mu_+$ and $\mu_+^\flat$ to $H_{x,0+}$ coincide with $\mu^\sharp |H_{x,0+}$.  In particular, $\mu_+|H_{x,0+}$ is independent of the choice of weak factorization.
\end{corollary}

Once again, we emphasize that in the definitions and results in this section we have assumed that $p$ does not divide the order of $\pi_1(\bG_{\rm der})$.

\subsection{Constructing regular supercuspidal represen\-ta\-tions: first approach}\label{sec:first}

Properly speaking, instead of saying that a pair $(\bT,\mu)$ is a ``tame, elliptic, regular pair,'' one should say that it is ``tame, elliptic, regular pair relative to $\bG$.''  Indeed, if one replaces $\bG$ by any of the associated subgroups $\bG^i$ then $(\bT,\mu)$ may be viewed as a tame, elliptic, regular pair relative to $\bG^i$.  Accordingly,  Kaletha's theory associates to each tame, elliptic, regular pair $(\bT,\mu)$ relative to $\bG$, a sequence 
$$\pi_0,\dots, \pi_d$$ of supercuspidal representations of $G^0,\dots,G^d$, respectively.
  The depth of each $\pi_i$ is the same as the depth of $\mu$.   The equivalence class of $\pi_i$ is canonically associated to $\mu$.
  
This is not the same as Yu's sequence of representations, but there is a simple relation, discussed below, between the two sequences.  Yu's sequence depends on the Howe factorization.

There is a natural way to associate to $\mu$  an irreducible representation $\rho_T^\mu$ of $TH_{x,0}$.   When $\bT$ is anisotropic and $\mu$ has depth zero, then $TH_{x,0}= H_{x,0}$ and $\rho_T^\mu$ is just the pullback to $H_{x,0}$ of the Deligne-Lusztig irreducible cuspidal representation of $\mathsf{H}_x^\circ (\mathfrak{f}_F) = H_{x,0:0+}$ associated to the character of $\mathsf{T}(\mathfrak{f}_F)= T_{0:0+}$ associated to $\mu$.
In general, when $\mu$ has depth zero, the definition of $\rho_T^\mu$ is given in \cite[\S3.4.4]{KalYu}.  (The notation $\tilde\kappa_{(S,\theta)}$ is used there.)

We will discuss the construction of $\rho_T^\mu$ for arbitrary depths in full detail later in this paper.

The first  supercuspidal representation in the sequence  $\pi_0,\dots ,\pi_d$ is given by
$$\pi_0 = \pi_0 (\mu) = {\rm ind}_{TH_{x,0}}^H(\rho^\mu_T).$$
When $\mu$ has depth zero, we have $H= G$ and  $\pi = \pi(\mu)= \pi_0$.

In \S3.7 of \cite{KalYu}, Kaletha defines what it means for a sequence $\mu_{-1},\dots , \mu_d$ to be a ``Howe factorization'' of $\mu$, and he proves the existence of Howe factorizations.  Once we are given a Howe factorization, if we take
$$\vec\mu = (\mu_0,\dots , \mu_d),$$  then  the triple
$$\Psi_{\vec\mu} = \left( \vec\bG , {\rm ind}_{TH_{x,0}}^H(\rho_T^{\mu_{-1}}), \vec\mu \right)$$
is a datum of the type used by Yu \cite[\S3]{MR1824988} in his construction.
Kaletha's representation $\pi (\mu)$ is Yu's representation $\pi (\Psi_{\vec\mu})$.

The equivalence class of $\pi (\mu)$ does not depend on the choice of Howe factorization of $\mu$.  To construct the other members of the sequence $\pi_0,\dots, \pi_d$, one uses the same character $\mu$, but one replaces $\bG$ by each of the other groups $\bG^i$.

By contrast, the $i$-th representation in Yu's sequence is attached to $\bG^i$ and the character $\prod_{j=-1}^i (\mu_j | T)$ of $T$.  So, other than $\pi= \pi_d$, the representations in Yu's sequence depend on the Howe factorization.

\subsection{Constructing regular supercuspidal represen\-ta\-tions: second approach}\label{sec:second}

Let us now describe in rough terms  how to construct $\pi(\mu)$ following the  alternative to Yu's construction outlined in \cite[\S2]{ANewYu}.    Let $H_x$ be the stabilizer in $H$ of the point $x$ (which we recall is a vertex in the reduced building of $\bH$).  Let $$\rho  = \rho(\mu) ={\rm ind}_{TH_{x,0}}^{H_x}(\rho_T^\mu).$$

These representations $\rho$ of $H_x$ are examples of the ``permissible representations''  that appear in \cite[Definition 2.1.1]{ANewYu}.  Permissible representations (as opposed to Kaletha's characters $\mu$)  are the atomic particles that parametrize the supercuspidal representations in \cite{ANewYu}.

It turns out that if one restricts the permissible representation $\rho$ to $H_{x,0+}$, then $$\rho |H_{x,0+} = \phi \cdot {\rm Id},$$ for some (canonical) character $\phi$ of $H_{x,0+}$.  In other words, $\rho |H_{x,0+}$ is $\phi$-isotypic.  The character $\hat\phi$ of $K_+$ mentioned above is a canonical extension of $\phi$ to $K_+$.  (See \S2.6 and \S3.9 \cite{ANewYu}.)

In place of Howe factorizations, the theory in \cite{ANewYu} employs  additional invariants that represent the canonical common essence of the various factorizations.  
We now provide a quick description of them.

Consider the restrictions $\phi |Z^{i,i+1}_{r_i}$, where
$$Z^{i,i+1}_{r_i}= (\bG^{i+1}_{\rm der}\cap \bZ^i)^\circ(F)_{r_i},$$
 $\bG^{i+1}_{\rm der}$ is the derived group of $\bG^{i+1}$, and $\bZ^i$ is the center of $\bG^i$.   
Using standard duality theory, one can equate such restrictions with clusters of elements in the dual of the Lie algebra of $\bG^{i+1}$.  This is done in \cite[\S2.7]{ANewYu}.  The cluster of elements   is referred to as  {\it the dual coset of $\phi |Z^{i,i+1}_{r_i}$}  and it is denoted by $(\phi |Z^{i,i+1}_{r_i})^*$.

 The elements of the dual coset are the generic elements that are used to construct certain Weil representations, analogous to the situation in the Howe/Yu constructions.  Restricting each of the latter representations, we obtain representations $\omega_i$ of $H_x$ for $0\le i\le d-1$.  (More precisely, the dual coset $(\phi |Z^{i,i+1}_{r_i})^*$ gives rise to a Weil representation that is defined on a certain finite symplectic group.  Embedded in the symplectic group is a subgroup that is a quotient of $H_x$.  The restriction of the Weil representation to this subgroup pulls back to $\omega_i$.)

To specify the equivalence class of our inducing representation $\kappa$ of $K$, it suffices to define $\kappa$ on the support of its character.
This can be done canonically.

First, we define $\kappa$ on $H_x$ to be the tensor product
$${\rm ind}_{TH_{x,0}}^{H_x} (\rho_T^\mu) \otimes \omega_0 \otimes \cdots \otimes \omega_{d-1}$$ of representations of $H_x$.
Next, we define $\kappa$ on $K_+$ to be a multiple of the character $\hat\phi$.
Since the character of $\kappa$ is supported in $H_xK_+$, the latter conditions completely determine the character, and hence the equivalence class, of $\kappa$.

In \S3.11 \cite{ANewYu}, it is shown that a representation $\kappa$ with the prescribed character actually exists.

\subsection{A simplified construction of the representation $\rho_T^\mu$ of $TH_{x,0}$}\label{sec:simplified}

The Deligne-Lusztig construction of a virtual representation of a finite group of Lie type requires a choice of a Borel subgroup.  But this choice is not evident in the Deligne-Lusztig character formula.  So, while the virtual representation depends on the choice of the Borel subgroup,  its equivalence class does not. 

Similarly, in this section, we use a choice of weak factorization to construct $\rho_T^\mu$, and later we show  (via 
the character formula in Proposition \ref{DLformula}) that the equivalence class of $\rho_T^\mu$ is independent of the weak factorization.

Suppose $(\bT,\mu)$ is a tame, elliptic, regular pair.  Using Proposition 3.7.4 \cite{KalYu}, Kaletha chooses a Howe factorization $\mu_{-1},\dots ,\mu_d$ of $\mu$ and then constructs in \S3.4.3 \cite{KalYu} a representation $\rho_T^{\mu_{-1}}$  (denoted by $\tilde\kappa_{(S,\theta)}$ in \cite{KalYu}).  Then he follows the approach sketched in \S\ref{sec:second} above.

Our modification of Kaletha's approach uses weak factorizations $(\mu_-,\mu_+)$ of $\mu$.  Again, we need to construct $\rho_T^{\mu_{-}} = \rho_T^{\mu_{-1}}$, but we use a simplified approach.  
Once $\rho_T^{\mu_{-}}$ is defined, we take
$$\rho_T^\mu = \rho_T^{\mu_{-}}\otimes   (\mu_+\,  |\,  TH_{x,0}) .$$

When $p$ divides the order of $\pi_1(\bG_{\rm der})$, we replace $\bG$ with $\bG^\sharp/\bN$, as discussed above in \S\ref{sec:zeroreduction} and Remark 2.1.2 in \cite{ANewYu}.
The representation $\rho_T^\mu$ is defined over the $z$-extension, again as a product of a depth zero representation (in the obvious sense) and a positive depth character.  When one descends from $\bG^\sharp$ back to $\G$, the factors in the definition of $\rho_T^\mu$ are not necessarily defined over $\bG$ (in other words, they may not be trivial on $\bN$), but the product is well-defined.

The construction of $\rho_T^{\mu_{-}}$ starts with the Deligne-Lusztig construction of irreducible, cuspidal representations of finite groups of Lie type.  According to
Fact 3.4.11 \cite{KalYu}, the character $\mu_{-}|T_0$ factors to a general position character $\bar\mu_{-}$ of the  torus $\mathsf{T}(\mathfrak{f}_F) = T_{0:0+}$ in $\mathsf{H}_x^\circ (\mathfrak{f}_F)= H_{x,0:0+}$.  Adjusting the sign of the associated Deligne-Lusztig virtual representation $R_{\mathsf{T}}^{\bar\mu_{-}}$, we obtain an irreducible cuspidal representation $(-1)^{\ell (w)} R_{\mathsf{T}}^{\bar\mu_{-}}$ of $\mathsf{H}^\circ_x(\mathfrak{f}_F)$, where $w$ is the Weyl group element that corresponds to the $\mathsf{H}_x^\circ(\mathfrak{f}_F)$-conjugacy class of $\mathsf{T}$, as in Corollary 1.14 \cite{MR0393266}.  (See \S4.2 \cite{ANewHM} for another interpretation of these signs.)

The restriction of $\rho_T^{\mu_{-}}$ to $H_{x,0}$ is defined to be the pullback of $(-1)^{\ell (w)} R_{\mathsf{T}}^{\bar\mu_{-}}$.

Recall that, in Deligne-Lusztig's construction, one must choose a Borel subgroup $\mathsf{B}$ that contains $\mathsf{T}$ (but is not stable under the Frobenius automorphism ${\rm Fr}$).  One then considers the $\ell$-adic cohomology with compact supports of the variety
$$\mathsf{X} = \mathsf{X}_{\mathsf{T}\subset \mathsf{B}} =  \{ h\in \mathsf{H}_x^\circ \ :\ h^{-1}{\rm Fr}(h)\in \mathsf{U}\},$$
where $\mathsf{U}$ is the unipotent radical of $\mathsf{B}$.

A basic property of $\ell$-adic cohomology that follows directly from the definitions is that  every automorphism of $\mathsf{X}$ induces a linear automorphism of $H^i_c(\mathsf{X},\overline{\Q}_\ell)$, for each $i$.  (See Property 7.1.3 \cite{MR1266626}.)

There are three relevant group actions on $\mathsf{X}$ that yield actions on cohomology.

First, there is the action of $\mathsf{T}(\mathfrak{f}_F)$ on $\mathsf{X}$ by right translations.  Let
$$H^i_c(\mathsf{X},\overline\Q_\ell)_{\bar\mu_{-}}= \left\{ v\in H^i_c(\mathsf{X},\overline\Q_\ell) \ :\ v\cdot t = \bar\mu_{-} (t)v,\ \forall t\in \mathsf{T}(\mathfrak{f}_F) \right\}.$$

The Deligne-Lusztig representation spaces are defined as alternating sums involving the  spaces $H^i_c(\mathsf{X},\overline\Q_\ell)_{\bar\mu_{-}}$.  Thanks to a result of He \cite{MR2427638},  in our case, the sums have precisely one nonzero term, namely, the $i= (-1)^{\ell (w)}$ term.  

So the representation space of $\rho_T^{\mu_{-}}$ is the space
$$\mathsf{V}_{\bar\mu_{-}} = H^{\ell (w)}_c(\mathsf{X},\overline\Q_\ell)_{\bar\mu_{-}}.$$

The second group action on $\mathsf{X}$ is the action of $\mathsf{H}_x^\circ (\mathfrak{f}_F)$ by left translations.   
The representation $(-1)^{\ell (w)} R_{\mathsf{T}}^{\bar\mu_{-}}$ is the corresponding linear action of $\mathsf{H}_x^\circ (\mathfrak{f}_F)$ on $\mathsf{V}_{\bar\mu_{-}}$.

The third group action on $\mathsf{X}$ is an action of $T$ that was defined in \S3.4.4 \cite{KalYu}.  We now define this action in a simpler, but equivalent, way.

Recall that $\mathsf{X}$ is a subvariety of $\mathsf{H}_x^\circ$.  Let us view the elements $\mathsf{H}_x^\circ$ as cosets via the identification $\mathsf{H}_x^\circ = \bH (F^{\rm un})_{x,0:0+}$.  Since $T$ fixes $x$, the group $T$ normalizes both $\bH (F^{\rm un})_{x,0}$ and $\bH (F^{\rm un})_{x,0+}$, and hence it acts by conjugation on $\mathsf{H}_x^\circ$.  

\begin{lemma}
The conjugation action of $T$ on $\mathsf{H}_x^\circ$ stabilizes $\mathsf{X}$.
\end{lemma}

\begin{proof}
Suppose $t\in T$ and suppose the coset $c= h\bH(F^{\rm un})_{x,0+}$ lies in $\mathsf{X}$.
Let $\widetilde{\mathsf{U}}$ be the preimage of $\mathsf{U}$ in $\bH(F^{\rm un})_{x,0}$.
  To say that $c^{-1}{\rm Fr}(c)$ lies in $\mathsf{U}$ means that the corresponding coset is a subset of $\widetilde{\mathsf{U}}$.

We observe that $tht^{-1}\in \bH(F^{\rm un})_{x,0}$ and 
$$t(h\bH(F^{\rm un})_{x,0+})t^{-1}=
tht^{-1}\bH(F^{\rm un})_{x,0+}\in \mathsf{H}_x^\circ.$$
So we have a well-defined element $tct^{-1} = tht^{-1}\bH(F^{\rm un})_{x,0+}$ of $\mathsf{H}_x^\circ$.

We need to show that $(tct^{-1})^{-1} {\rm Fr}(tct^{-1})$ lies in $\mathsf{U}$.  In other words, from the coset point of view, we need that the {\it set}
$(tct^{-1})^{-1} {\rm Fr}(tct^{-1})$ is a subset of the preimage of $\mathsf{U}$ in $\bH(F^{\rm un})_{x,0}$.

Using the natural identification 
$${\rm Gal}(F^{\rm un}/F)
= {\rm Gal}({\mathfrak{f}_{F^{\rm un}}}/
\mathfrak{f}_F),
$$
we see that $$(tct^{-1})^{-1} {\rm Fr}(tct^{-1})= tc^{-1}{\rm Fr}(c)t^{-1} .$$
On the right hand side, $c^{-1}{\rm Fr}(c)$ is viewed as a coset.  Since $c\in \mathsf{X}$, we have $c^{-1}{\rm Fr}(c)\subset \widetilde{\mathsf{U}}$.
But $\widetilde{\mathsf{U}}$ must be normalized by $T$, since it is normalized by $\bT (F_{\rm un})_0$ and the center of $\bH$, and since $\bT$ is $F$-elliptic.  Therefore,
$tc^{-1}{\rm Fr}(c)t^{-1}\subset \widetilde{\mathsf{U}}$ and thus $tct^{-1}\in \mathsf{X}$.
\end{proof}

So the third action on $\mathsf{X}$ is the action of $T$ by conjugation.  This yields an action of $T$ on $\mathsf{V}_{\bar\mu_{-}}$.  When $t\in T$ and $v\in \mathsf{V}_{\bar\mu_{-}}$, we use the notation ${\rm Int}(t)v$ to denote this action.
The definition of $\rho_T^{\mu_{-}}$ on $T$ is given by
$$\rho_T^{\mu_{-}} (t) = \mu_{-}(t)\ {\rm Int}(t).$$
It follows that the restriction of $\rho_T^\mu$ on $T$ is defined by
$$\rho_T^{\mu} (t) = \mu(t)\ {\rm Int}(t).$$

\subsection{From $TH_{x,0}$ to $H_x$}

As discussed above, the supercuspidal representations  in \cite{KalYu} are constructed from certain representations of $H_x$ called ``permissible representations''  \cite[Definition 2.1.1]{ANewYu}.
For regular supercuspidal representations, the relevant permissible representation is defined by
$$\rho = {\rm ind}_{TH_{x,0}}^{H_x}(\rho_T^\mu).$$

\begin{lemma} $\rho$ is a permissible representation.\end{lemma}\label{permiss}

\begin{proof}
Our first task is to show that the representation $$\pi_{0} = {\rm ind}_{TH_{x,0}}^H (\rho_T^\mu) = {\rm ind}_{H_x}^H(\rho)$$
is irreducible.

Let $(\mu_-,\mu_+)$ be a weak factorization of $\mu$.
Then $$\pi_{0}\simeq  {\rm ind}_{TH_{x,0}}^H (\rho_T^{\mu_{-}}) \otimes \mu_+ .$$
Thus  irreducibility of $\pi_{0}$  is equivalent to the irreducibility of the representation ${\rm ind}_{TH_{x,0}}^H (\rho_T^{\mu_{-}})$, which follows from 
 Lemma 3.4.12 \cite{KalYu}.  (Kaletha's lemma is a  consequence of Proposition 6.6 \cite{MR1371680} and a straightforward modification of the proof of Lemma 4.5.1 of \cite{MR2660683}.)
 
The next step is to show that the restriction of $\rho$ to $H_{x,0+}$ is a multiple of some character $\phi$.  But it is elementary to verify this with $\phi = \mu^\sharp |H_{x,0+}$, where $\mu^\sharp$ is the character of $TH_{x,0+}$ defined in Definition \ref{musharp}.

The fact that $\phi$ is trivial on the group denoted as $H_{{\rm der},x,0+}^\flat$ in \cite{ANewYu} follows from our assumption on $p$ and Lemma 3.2.1 \cite{ANewYu}.

Finally, since we assume $p$ is not a bad prime, \cite[Lemma 8.1]{MR1824988} implies Condition (4) in \cite[Definition 2.1.1]{ANewYu} is satisfied.
\end{proof}

\section{The character of $\rho_T^\mu$}\label{sec:rhoTmu}

\subsection{Topological Jordan decompositions for compact-mod-center groups}\label{sec:topJoe}


 In Proposition \ref{DLformula}, we give a formula for the character of $\rho_T^\mu$ that extends 
the Deligne-Lusztig (virtual) character formula \cite[Theorem 4.2]{MR0393266}.
The Deligne-Lusztig formula expresses the value of a virtual character at a given point in terms of that element's Jordan-Chevalley decomposition.  Similarly, our formula espresses the character of $\rho_T^\mu$ at an element of $TH_{x,0}$ in terms of a variant of the topological Jordan decomposition for compact-mod-center groups.  We develop the theory of such decompositions in this section.  The essential elements of this theory can be found in  \cite{MR2408311}.

Let $\bG$ be a connected reductive group that is defined over a field $F$ that is a finite extension of a field $\Q_p$ of $p$-adic numbers.
We would like to extend the theory of topological Jordan decompositions  to a slightly broader class of $F$-rational elements $g\in G = \bG (F)$.

Our objective is not to be as general as possible, but rather to give the most precise theory that fits the applications we have in mind.  We refer to Spice's article \cite{MR2408311} for a  general discussion of the theory of topological Jordan decompositions, which builds on the earlier works \cite[p.~226]{MR748509}, \cite[p.~213]{MR1000107}, and \cite[pp.~112-113]{MR1216184}.

Consider the closure $K_g$ of the subgroup of $G$ generated by $g$.  We say that $g$ is {\bf compact} if $K_g$ is compact.
Compact elements are precisely the elements that admit (standard) topological Jordan decompositions.
A particularly attractive way to define the topological Jordan decomposition is as follows.  Suppose $g$ is compact.  Then $K_g$ is a profinite abelian group and thus has a  direct product decomposition
$$K_g = \prod_\ell K_{g,\ell},$$ where for each prime $\ell$ we let $K_{g,\ell}$ denote the (unique) pro-$\ell$ Sylow subgroup of $K_g$.
The product
$$K_{g,p'} = \prod_{\ell\ne p} K_{g,\ell}$$ is known to be finite.  (See \cite[Remark 1.9]{MR2408311}.)  (Recall that $p$ is the characteristic of the residue field of $F$.) 
Let $u$ denote the $K_{g,p}$-component of $g$ and let $s = gu^{-1}$ be the $K_{g,p'}$-component.   Then $g= su=us$ is the {\bf topological Jordan decomposition} of $g$.

The elements that arise as $u$'s are said to be {\bf topologically unipotent} and the elements that arise as $s$'s are called {\bf absolutely semisimple}.

Topological unipotent elements are connected with unipotent elements in  the reductive quotients associated to points in the Bruhat-Tits buildings associated to $\bG$. According to Lemma 2.30 \cite{MR2408311}, $g$ is topologically unipotent precisely when there exists a  finite extension $E$ of $F$ and a point $y$ in the extended Bruhat-Tits building $\mathscr{B} (\bG , E)$ such that the prounipotent radical $\bG (E)_{y,0+}$ of the parahoric subgroup $\bG (E)_{y,0}$ contains $g$.  
If $x$ is another point in $\mathscr{B} (\bG , E)$ then 
the image of $\bG(E)_{x,0}\cap \bG (E)_{y,0+}$ in the reductive quotient
$$\mathsf{G}_x^\circ (\mathfrak{f}_E) = \bG(E)_{x,0:0+}$$ is the group $\mathsf{U}(\mathfrak{f}_E)$ of $\mathfrak{f}_E$-rational points of the unipotent radical $\mathsf{U}$ of a parabolic subgroup $\mathsf{P}$ of $\mathsf{G}_x^\circ$.  In particular, $\mathsf{P}$ is  the parabolic subgroup such that $\mathsf{P}(\mathfrak{f}_E)$ is the image of $\bG(E)_{x,0}\cap \bG (E)_{y,0}$ in  
$\mathsf{G}_x^\circ (\mathfrak{f}_E)$.

Every maximal compact subgroup of $G$ is the stabilizer ${\rm stab}_G(x)$ of some point $x$ in $\mathscr{B}(\bG, F)$.  If $g$ is compact then $K_g$ must therefore be contained in some ${\rm stab}_G(x)$.  On the other hand, given $x\in \mathscr{B}(\bG,F)$, whether or not ${\rm stab}_G(x)$ is maximal compact, it is compact and one can consider the topological Jordan decomposition of its elements.
Spice shows that if $g\in {\rm stab}_G (x)$ then the topological Jordan decomposition  of $g$ projects to Jordan-Chevalley decomposition in the reductive quotient
$\mathsf{G}^\circ_x (\mathfrak{f}_F) = G_{x,0:0+}$.  (See Lemma 2.30 and Theorem 2.38 in \cite{MR2408311}.)

Given a point $x$ in the extended building $\mathscr{B}(\bG,F)$, the associated parahoric subgroup and its prounipotent radical only depend on the image of $x$ in the reduced building $\mathscr{B}_{\rm red}(\bG,F)$.  So in our notations related to parahoric subgroups, such as ``$G_{x,0}$,'' we allow the index $x$ to represent either a point in the extended or reduced building.  Given $x\in \mathscr{B}_{\rm red}(\bG, F)$, we let $G_x$ denote the stabilizer of $x$ in $G$.  This is a compact-mod-center subgroup.

We now develop topological Jordan decompositions for elements $g$ of $G_x$.

The first step is to take topological Jordan decompositions modulo the split component of the center of $G$, as in \cite{MR2408311}.  Let $\bA_\bG$ be the largest $F$-split torus in the center of $\bG$ and let $\overline\bG= \bG/\bA_\bG$.  Since $\bA_\bG$ is split, $H^1(F,\bA_\bG)$ is trivial and thus 
$$\overline G  = G/A_G,$$ where $\overline G = \overline{\bG}(F)$ and $A_G = \bA_\bG(F)$. 
The reduced buildings of $\bG$ and $\overline{\bG}$ are naturally identified, and  when $x\in \mathscr{B}_{\rm red}(\bG,F)$  we have $$\overline G_x = G_x/A_G .$$
Note that the latter group is compact, and hence all of its elements have topological Jordan decompositions.
Given $g\in \bG$, let $\bar g$ denote the image of $g$ in $\overline{\bG}$.
Let $\bar g = s_{\bar g}u_{\bar g}$ denote the topological Jordan decomposition  in $\overline G_x$ of $\bar g$.

\begin{definition} Given $x\in \mathscr{B}_{\rm red}(\bG,F)$ and $g\in G_x$, a  {\bf topological Jordan decomposition} of $g$ is a decomposition $g= su=us$ such that $\bar s$ is absolutely semisimple and $u$ is topologically unipotent.
\end{definition}

We observe that the group $(A_G)_{0+}$ on the set of topological Jordan decompositions of the latter type according to $$a\cdot (s,u) = (as,a^{-1}u).$$

The next result is the main result of this section and it is an immediate consequence of the two lemmas that follow it.  The lemmas also provide a bit of extra detail.

\begin{proposition}
Given $x\in \mathscr{B}_{\rm red}(\bG,F)$, every element of $G_x$ admits a topological Jordan decomposition.  Such decompositions are unique modulo the action of $(A_G)_{0+}$.
\end{proposition}

The latter result extends Spice's main result on topological Jordan decompositions \cite[Theorem 2.38]{MR2408311} and it is compatible with Kaletha's topological Jordan decompositions of strongly regular elements  \cite[Lemma 3.4.13]{KalYu}.

\begin{lemma}
Suppose $x\in \mathscr{B}_{\rm red}(\bG,F)$ and $g\in G_x$.    If $s$ and $u$ are arbitrary preimages of $s_{\bar g}$ and $u_{\bar g}$, respectively, in $G_x$ then $s$ and $u$ lie in the abelian group $K_gA_G$ and hence they must commute.
\end{lemma}

\begin{proof}
The natural homomorphism $K_g\to K_{\bar g}$ of abelian groups is a surjection with kernel $A_G\cap K_g$.  Choose preimages $s_0$ and $u_0$ of $s_{\bar g}$ and $u_{\bar g}$  in $K_g$.  Since $K_g$ is abelian, $s_0$ and $u_0$ must commute.  If $s$ and $u$ are preimages of $s_{\bar g}$ and $u_{\bar g}$ in $G_x$ then there exist elements $a,a'\in A_G$ such that $s=as_0$ and $u= a's_0$ or, in other words, $s$ and $u$ lie in $K_gA_G$.\end{proof}

\begin{lemma}
The image in $\overline G$ of the set of topological unipotent elements in $G$ is precisely the set of topological unipotent elements in $G$.  
Consequently, if $x\in \mathscr{B}_{\rm red}(\bG,F)$ then every topologically unipotent element in $\overline G_x$ is the image of a topologically unipotent element in $G_x$.  Given $g\in G_x$, the set of topologically unipotent preimages of $u_{\bar g}$ in $G_x$ comprises a single coset in $(K_gA_G)/(A_G)_{0+}$.
\end{lemma}

\begin{proof}
If $E$ is a finite extension of $F$ and $y$ is a point in $\mathscr{B}_{\rm red}(\bG,E)$ then the natural homomorphism $\bG(E)_{y,0+}\to \overline\bG(E)_{y,0+}$ is surjective.  (See \cite[Lemma 3.3.2]{KalYu}.)  Taking unions over all $y$, we see that the image of the set $\bG(E)_{0+}$ is the set $\overline\bG(E)_{0+}$.  The first two claims  follow.

Now suppose $g\in G_x$.  We know that since $u_{\bar g}$ is topologically unipotent it has topologically unipotent preimages in $G$.  But all such preimages must lie in $K_gA_G$.  If $u$ is one such preimage then any other such preimage must have the form $ua$, for some $a\in A_G$.  But when $a\in A_G$ then $ua$ is topologically unipotent precisely when $a\in (A_G)_{0+}$.  This completes the proof.\end{proof}

\subsection{A formula for the character of $\rho_T^\mu$}

In this section, we give a formula for the trace of $\rho_T^\mu(k)$, where $k$ is an element of $TH_{x,0}$ with topological Jordan decomposition $k = su$.  Explicitly, the formula
looks like
$${\rm trace} (\rho_T^\mu(k)) = \frac{(-1)^{\ell(w)}}{|\mathsf{M}_s(\mathfrak{f}_F)|}\sum_{h\in \mathsf{H}_x^\circ (\mathfrak{f}_F)}{}^h\dot\mu(s)\ \hat\mu (u)\ Q_{h\mathsf{T}h^{-1}}^{\mathsf{M}_s}(u),$$ and it may be viewed as an extension of the Deligne-Lusztig character formula \cite[Theorem 4.2]{MR0393266}.  In the depth zero case, it is a variant of  character formulas in \cite{KalYu}.

Before proving the formula, the first order of business is to explain the notations.
As usual, there is no loss in generality in assuming that $p$ does not divide the order of $\pi_1(\bG_{\rm der})$, and so we do so.

We start by explaining the meaning of 
${}^h\dot\mu (s)$.

In Definition \ref{musharp},  we have defined a canonical extension of the character $\mu$ of $T$ to a character 
$\mu^\sharp$ of $TH_{x,0+}$.
Let us extend $\mu^\sharp$ to a function $\dot\mu :TH_{x,0}\to\C$ by taking
$\dot\mu \equiv 0$ on the complement of $TH_{x,0+}$ in $TH_{x,0}$.
When $h\in H_{x,0}$,  let 
$${}^h\dot\mu(s) = \dot\mu (h^{-1}sh).$$
Since ${}^h\dot\mu = \dot\mu$ whenever $h\in H_{x,0+}$, there is a natural interpretation of ${}^h\dot\mu$ when $h\in \mathsf{H}_x^\circ (\mathfrak{f}_F) = H_{x,0:0+}$.

We also extend the character $\mu^\sharp$ in another way.  Let $\mathscr{U}$ be the set of topologically unipotent elements in $H_{x,0}$.  We can extend $\mu^\sharp$ to a function 
$$\hat\mu : T\mathscr{U}\to \C$$
by taking $$\hat\mu(tu) = \mu(t) \mu_+(u),$$ whenever $t\in T$ and $u\in\mathscr{U}$.
One shows that this is well-defined and independent of the weak factorization just as in the proof of Lemma \ref{mushlemma}.

Next, we explain the notation $\mathsf{M}_s$.
Roughly speaking, $\mathsf{M}_s$ is the connected centralizer in $\mathsf{H}_x^\circ$ of $s$.  A more precise description is given as follows.  The connected centralizer $Z_\bH^\circ (s)$ of $s$ in $\bH$ is a connected reductive $F$-group.  The group $$\bH (F^{\rm un})_{x,0}\cap Z_\bH^\circ (s)(F^{\rm un})$$ is a parahoric subgroup of $Z_\bH^\circ (s)(F^{\rm un})$ with pro-unipotent radical $$\bH (F^{\rm un})_{x,0+}\cap Z_\bH^\circ (s)(F^{\rm un}).$$
We take $\mathsf{M}_s$ to be the corresponding reductive quotient (that is, the quotient of the latter two groups), and it may be viewed as a (Levi) subgroup of $\mathsf{H}_x^\circ = \bH(F^{\rm un})_{x,0:0+}$.

Finally, we explain the notation $Q_{h\mathsf{T}h^{-1}}^{\mathsf{M}_s}(u)$.
In our putative character formula, the only relevant pairs $(s,h)$ are those for which ${}^h\dot\mu (s)$ is nonzero.  We claim that for such $(s,h)$, it must be the case that $h\mathsf{T}h^{-1}$ is a maximal $\mathfrak{f}_F$-torus in $\mathsf{M}_s$.
It suffices to show that if $\hat h$ is a preimage of $h$ in $H_{x,0}$ then $\hat h\bT(F^{\rm un})_0 \hat h^{-1}$ is contained in $\bH (F^{\rm un})_{x,0}\cap Z_\bH^\circ (s)(F^{\rm un})$.
There is no loss in generality in assuming that $h$ and $\hat h$ are trivial and  $s\in T$, in which case our claim is obvious.  It follows that we have a well-defined Green function $Q_{h\mathsf{T}h^{-1}}^{\mathsf{M}_s}$ on the unipotent set in $\mathsf{M}_s$.  (See Definition 4.1 \cite{MR0393266}.)
When we write $Q_{h\mathsf{T}h^{-1}}^{\mathsf{M}_s}(u)$ for a prounipotent element $u$, we really mean $Q_{h\mathsf{T}h^{-1}}^{\mathsf{M}_s}(\bar u)$, where $\bar u$ is the corresponding unipotent element in $\mathsf{M}_s(\mathfrak{f}_F)$.

Having established our notations, we can officially state our character formula:

\begin{proposition}\label{DLformula}  If $k\in TH_{x,0}$ has topological Jordan decomposition $k=su$ then
$${\rm trace} (\rho_T^\mu(k)) = \frac{(-1)^{\ell(w)}}{|\mathsf{M}_s(\mathfrak{f}_F)|}\sum_{h\in \mathsf{H}_x^\circ (\mathfrak{f}_F)}{}^h\dot\mu(s)\ \hat\mu (u)\ Q_{h\mathsf{T}h^{-1}}^{\mathsf{M}_s}(u).$$
\end{proposition}

Since the proof is rather involved, we first sketch the main structure and establish some key details in an auxiliary lemma.

The proof begins with an application of the Deligne-Lusztig fixed point formula to reduce from a formula involving the cohomology of $\mathsf{X}$ to a formula involving the cohomology of the variety 
$$\mathsf{X}^{(s,t)} = \{ h\in \mathsf{X}\ :\ {\rm Int}(s)ht = h\},$$ where $s\in T$ and $t\in \mathsf{T}(\mathfrak{f}_F)$.
The second part of the proof involves showing that $\mathsf{X}^{(s,t)}$ is a finite disjoint union of a collection of open and closed subvarieties isomorphic to a common variety $\mathsf{Y}^{(s,t)}$.  The character is then expressed in terms of the cohomology of $\mathsf{Y}^{(s,t)}$.  The cohomology of $\mathsf{Y}^{(s,t)}$ is further manipulated until its  relation to the Green's functions $Q_{h\mathsf{T}h^{-1}}^{\mathsf{M}_s}$ becomes evident, at which point the desired character formula follows.

We turn now to the relation between $\mathsf{X}^{(s,t)}$ and $\mathsf{Y}^{(s,t)}$.
We need the following definitions:
\begin{eqnarray*}
\mathsf{W}^{(s,t)}&=&\{ h\in \mathsf{H}_x^\circ\ :\   {\rm Int}(s)ht=h\},\\
\mathsf{Z}^{(s,t)}&=&\{ h\in \mathsf{H}_x^\circ\ :\  {\rm Int}(s)h=tht^{-1} \},\\
\mathsf{Y}^{(s,t)}&=&\mathsf{X}\cap (\mathsf{Z}^{(s,t)})^\circ .
\end{eqnarray*}
Note that $\mathsf{W}^{(s,t)}$, $\mathsf{X}^{(s,t)}$, $\mathsf{Y}^{(s,t)}$, and $\mathsf{Z}^{(s,t)}$ are varieties over $\mathfrak{f}_{F^{\rm un}}$ and, in fact, $\mathsf{W}^{(s,t)}$ and $\mathsf{Z}^{(s,t)}$ are defined over $\mathfrak{f}_F$.
In addition, $\mathsf{Z}^{(s,t)}$ is a group and
$$\mathsf{X}^{(s,t)} = \mathsf{X}\cap \mathsf{W}^{(s,t)}.$$

\begin{lemma}\label{xst} The variety $\mathsf{X}^{(s,t)}$ is a finite disjoint union
$$\bigsqcup_{k\in \mathsf{W}^{(s,t)}(\mathfrak{f}_F)/(\mathsf{Z}^{(s,t)})^\circ (\mathfrak{f}_F)}k\mathsf{Y}^{(s,t)}$$ of open and closed subvarieties $k\mathsf{Y}^{(s,t)}$ isomorphic to $\mathsf{Y}^{(s,t)}$.  Moreover, $\mathsf{W}^{(s,t)}(\mathfrak{f}_F)$ is a $\mathsf{Z}^{(s,t)}(\mathfrak{f}_F)$-torsor, that is, if $k_0$ is any fixed element of $\mathsf{W}^{(s,t)}(\mathfrak{f}_F)$ then $$\mathsf{W}^{(s,t)}(\mathfrak{f}_F) = k_0\mathsf{Z}^{(s,t)}(\mathfrak{f}_F).$$
This gives an explicit decomposition of $\mathsf{X}^{(s,t)}$ as a disjoint union of $[\mathsf{Z}^{(s,t)}(\mathfrak{f}_F):(\mathsf{Z}^{(s,t)})^\circ (\mathfrak{f}_F)]$ open and closed subvarieties isomorphic to $\mathsf{Y}^{(s,t)}$.
\end{lemma}

\begin{proof}
The proof involves of a straightforward series of calculations that consists of verifying the following (ordered) steps:
\begin{itemize}
\item[(1)] $\mathsf{W}^{(s,t)}(\mathfrak{f}_F)\mathsf{Y}^{(s,t)}\subset \mathsf{X}^{(s,t)}$.
\item[(2)] $(\mathsf{Z}^{(s,t)})^\circ(\mathsf{f}_F)$ acts on $\mathsf{W}^{(s,t)}(\mathfrak{f}_F)\times \mathsf{Y}^{(s,t)}$ by
$m\cdot (k,z) = (km^{-1},mz)$.
\item[(3)] The orbits of the latter action are precisely the fibers of the multiplication map
$\mathsf{W}^{(s,t)}(\mathfrak{f}_F)\times \mathsf{Y}^{(s,t)}\to\mathsf{X}^{(s,t)}$.
\item[(4)] If $k_1,k_2\in \mathsf{W}^{(s,t)}(\mathfrak{f}_F)$ then $k_2^{-1}k_1\in \mathsf{Z}^{(s,t)}(\mathfrak{f}_F)$.
\item[(5)] If $k\in \mathsf{W}^{(s,t)}(\mathfrak{f}_F)$ and $z\in \mathsf{Z}^{(s,t)}(\mathfrak{f}_F)$ then $kz\in \mathsf{W}^{(s,t)}(\mathfrak{f}_F)$.
\item[(6)] If $h\in \mathsf{X}^{(s,t)}$ and $k\in \mathsf{W}^{(s,t)}(\mathfrak{f}_F)$ then there exists $z\in \mathsf{Z}^{(s,t)}(\mathfrak{f}_F)$ and $y\in \mathsf{Y}^{(s,t)}$ with $h= kzy$.
\item[(7)] $\mathsf{W}^{(s,t)}(\mathfrak{f}_F)\mathsf{Y}^{(s,t)} = \mathsf{X}^{(s,t)}$.
\end{itemize}
Each step follows directly from the definitions, the previous steps, and an obvious calculation.  We leave the details to the reader.
\end{proof}

\begin{proof}[Proof of Proposition \ref{DLformula}]
We assume $p$ does not divide the order of $\pi_1(\bG_{\rm der})$.  Once we prove our assertions under this assumption,  the general result follows from the discussion in \S\ref{sec:simplified}.

Choose a weak factorization $(\mu_-,\mu_+)$ of $\mu$.  Then 
$${\rm trace} (\rho_T^\mu(k)) = \mu_+(k) \cdot {\rm trace} (\rho_T^{\mu_-}(k))$$
and, similarly, the right hand side of the desired character formula is
$$ \mu_+(k) \cdot \frac{(-1)^{\ell(w)}}{|\mathsf{M}_s(\mathfrak{f}_F)|}\sum_{h\in \mathsf{H}_x^\circ (\mathfrak{f}_F)}{}^h\dot\mu_- (s)\ \hat\mu_- (u)\ Q_{h\mathsf{T}h^{-1}}^{\mathsf{M}_s}(u).$$
Accordingly, we now assume $\mu$ has depth zero or, in other words, $\mu=\mu_-$, since it suffices to establish this case.  In this case, $\hat\mu  (u)$ must, in fact, be trivial.

We  reduce things further by now assuming that $s$ lies in $T$.  This is justified by noting that both sides of the character formula are invariant under conjugation by $H_{x,0}$ and if ${\rm trace}(\rho^\mu_T(su))$ is nonzero then an $H_{x,0}$-conjugate of $s$ must lie  in $T$. 

At this point, our proof follows the approach in the proofs of \cite[Proposition 3.4.14]{KalYu}, \cite[Theorem 7.2.8]{MR1266626}, \cite[Theorem 6.8]{MR551499}, and \cite[Theorem 4.2]{MR0393266}.  Let $\tau$ be the right action of $\mathsf{T}(\mathfrak{f}_F)$ on $\mathsf{V}$ by right translations, that is,
$$\tau(t)v = v\cdot t,$$ as in \S\ref{sec:simplified}.  Define an idempotent operator $\iota$ on $\mathsf{V}$ by
$$\iota = \frac{1}{|\mathsf{T}(\mathfrak{f}_F)|} \sum_{t\in \mathsf{T}(\mathfrak{f}_F)}\bar\mu (t)^{-1}\ \tau(t).$$
Then the representation space of $\rho_T^\mu$ is the space
$$\mathsf{V}_{\bar\mu} = \iota \mathsf{V},$$ and we have a direct sum decomposition
$$\mathsf{V} = \mathsf{V}_{\bar\mu} \oplus (1-\iota)\mathsf{V}.$$
Since $\iota$ annihilates the second summand, we have
$${\rm tr}(\rho^\mu_T(su)) = (-1)^{\ell (w)}\mu(s)\ {\rm tr}({\rm Int}(s) R^{\bar\mu}_{\mathsf{T}}(u) \iota |\mathsf{V}),$$
where we are writing $R^{\bar\mu}_{\mathsf{T}}(u)$ for the Deligne-Lusztig operator $R^{\bar\mu}_{\mathsf{T}}(\bar u)$ associated to the image $\bar u$ of $u$ in $\mathsf{H}_x^\circ (\mathfrak{f}_F)$.
According to the definition of $\iota$,
$${\rm tr}(\rho^\mu_T(su)) = 
\frac{(-1)^{\ell (w)}\mu(s)}{|\mathsf{T}(\mathfrak{f}_F)|}
 \sum_{t\in \mathsf{T}(\mathfrak{f}_F)}\bar\mu (t)^{-1}
\ {\rm tr}(R^{\bar\mu}_{\mathsf{T}}(u) {\rm Int}(s) \tau(t) |\mathsf{V}).
$$
We now  apply the Deligne-Lusztig fixed point formula (Theorem 3.2 \cite{MR0393266}) to obtain
$${\rm tr}(\rho^\mu_T(su)) = 
\frac{(-1)^{\ell (w)}\mu(s)}{|\mathsf{T}(\mathfrak{f}_F)|}
 \sum_{t\in \mathsf{T}(\mathfrak{f}_F)}\bar\mu (t)^{-1}
\ {\rm tr}(R^{\bar\mu}_{\mathsf{T}}(u)  |H^{\ell(w)}_c(\mathsf{X}^{(s,t)},\overline\Q_\ell)).
$$

Property 7.1.6 \cite{MR1266626} (see also
\cite[page 120]{MR0393266}), when applied to the decomposition in Lemma \ref{xst}, yields
\begin{eqnarray*}{\rm tr}(\rho^\mu_T(su)) &=&
\frac{(-1)^{\ell (w)}\mu(s)}{|\mathsf{T}(\mathfrak{f}_F)|}
 \sum_{t\in \mathsf{T}(\mathfrak{f}_F)}\bar\mu (t)^{-1}\\
&&\sum_{h\in \mathsf{W}^{(s,t)}(\mathfrak{f}_F)/(\mathsf{Z}^{(s,t)})^\circ (\mathfrak{f}_F)} {\rm tr}(u  |H^{\ell(w)}_c(h\mathsf{Y}^{(s,t)},\overline\Q_\ell)).
\end{eqnarray*}

According to the definition of $\mathsf{W}^{(s,t)}$, the sums can be re-expressed as
$$ \sum_{t\in \mathsf{T}(\mathfrak{f}_F)} \sum_{h\in \mathsf{W}^{(s,t)}(\mathfrak{f}_F)/(\mathsf{Z}^{(s,t)})^\circ (\mathfrak{f}_F)}
=
\frac{1}{|(\mathsf{Z}^{(s,t)})^\circ (\mathfrak{f}_F)|}
\sum_{
\substack{
t\in \mathsf{T}(\mathfrak{f}_F)\\ h\in \mathsf{H}_x^\circ(\mathfrak{f}_F)\\ {\rm Int}(s)ht=h}
}.$$

It is now evident  that the sum over $t$ consists of at most one term, namely, the term associated to $t= ({\rm Int}(s)h^{-1})h$, assuming that the element $({\rm Int}(s)h^{-1})h$ actually lies in $\mathsf{T}(\mathfrak{f}_F)$.
For this value of $t$, we have
$(\mathsf{Z}^{(s,t)})^\circ = h^{-1}\mathsf{M}_sh$ and hence $\mathsf{Y}^{(s,t)} =h^{-1}\mathsf{M}_s h\cap  \mathsf{X}$.  The value of $t$ is defined precisely when $\hat h^{-1}s\hat h$ lies in $TH_{x,0+}$ for some, hence all, lifts $\hat h$ of $h$ in $H_{x,0}$.  Therefore, $\mu(s)\bar\mu (t)^{-1}$ can be replaced by ${}^h\dot\mu(s)$.
Next, we observe that
\begin{eqnarray*}
{\rm tr}(u  |H^{\ell(w)}_c(h\mathsf{Y}^{(s,t)},\overline\Q_\ell)) &=& {\rm tr}(h^{-1}uh  |H^{\ell(w)}_c(\mathsf{Y}^{(s,t)},\overline\Q_\ell))\\
&=& {\rm tr}(h^{-1}uh  |H^{\ell(w)}_c(h^{-1}\mathsf{M}_s h\cap  \mathsf{X},\overline\Q_\ell))\\
&=&|\mathsf{T}(\mathfrak{f}_F)|\cdot Q_{\mathsf{T}}^{h^{-1}\mathsf{M}_sh}(h^{-1}uh)\\
&=&|\mathsf{T}(\mathfrak{f}_F)|\cdot Q_{h\mathsf{T}h^{-1}}^{\mathsf{M}_s}(u).\end{eqnarray*}

Adjusting our formula according to all of these remarks yields the desired formula
$${\rm trace} (\rho_T^\mu(k)) = \frac{(-1)^{\ell(w)}}{|\mathsf{M}_s(\mathfrak{f}_F)|}\sum_{h\in \mathsf{H}_x^\circ (\mathfrak{f}_F)}{}^h\dot\mu(s)\ Q_{h\mathsf{T}h^{-1}}^{\mathsf{M}_s}(u).$$
\end{proof}

\bibliographystyle{amsalpha}
\bibliography{JLHrefs}

\end{document}